\documentclass{article}
\usepackage{amsmath}
\usepackage{amsfonts}
\usepackage{amsthm}
\usepackage{amssymb}
\usepackage{mathrsfs}
\usepackage{hyperref}

\theoremstyle{definition}
\newtheorem{definition}{Definition}
\newtheorem{example}{Example}

\theoremstyle{remark}
\newtheorem*{remark}{Remark}

\theoremstyle{plain}
\newtheorem{theorem}{Theorem}
\newtheorem{corollary}{Corollary}
\newtheorem{lemma}{Lemma}
\newtheorem{prop}{Proposition}

\DeclareMathOperator{\im}{Im}
\DeclareMathOperator{\re}{Re}

\DeclareMathOperator{\defeq}{\stackrel{def^{\underline{n}}}{=}}

\begin{document}

\title{Laplace--Carleson embeddings and weighted infinite-time admissibility}
\author{Andrzej S. Kucik}
\date{}
\maketitle

\begin{flushright}
School of Mathematics \\
University of Leeds \\
Leeds LS2 9JT \\
United Kingdom \\
e-mail: \ttfamily{mmask@leeds.ac.uk}

\end{flushright}

\begin{abstract}
In this paper we will establish necessary and sufficient conditions for a Laplace--Carleson embedding to be bounded for certain spaces of functions on the positive half-line. We will use these results to characterise weighted (infinite-time) admissibility of control and observation operators. We present examples of weighted admissibility criterion for one-dimensional heat equation with Neumann boundary conditions, and a~cetrain parabolic diagonal system which was previously known to be not admissible in the~unweighted setting.
\end{abstract}

\section{Introduction}
Although in most common physical applications the norms used are usually the unweighted $L^1$ or $L^2$ norms, it may sometimes be useful to consider weighted $L^p$ norms. The main purpose of this paper is to generalise admissibility criteria, obtained in \cite{jacob2014} for weighted $L^2$- and unweighted $L^p$-admissibility (given in terms of Carleson measures and Carleson-Laplace embeddings described in \cite{jacob2013}) to weighted $L^p(0, \, \infty)$ case, applying and generalising recent results from \cite{kucik2017}, concerning the spaces defined and studied in \cite{kucik2016}. A powerful boundedness criterion for the Laplace--Carleson embeddings for weighted $L^p$ spaces, containing the earlier version from \cite{jacob2013} as a special case, is also proved here.

This article is structured as follows. In Section \ref{sec:admissibility} the theory and definitions of admissibility for diagonal semigroups are outlined. Two important theorems, linking admissibility to Laplace--Carleson embeddings, are also cited there. In Section \ref{sec:kernels} some results concerning reproducing kernel Hilbert spaces are given. As a special case, the definitions of so-called Zen spaces and their generalisation are provided, presenting their connection to the admissibility concept. In Section \ref{sec:Ap} boundedness of Carleson embeddings for these generalised spaces is studied, following a similar analysis from \cite{kucik2017}. And finally, in Section \ref{sec:sectorial} boundedness of Laplace--Carleson embeddings for sectorial measures is characterised there, and we believe that Theorem \ref{thm:sectorial} is the most important result of this paper. This theorem is followed by two examples illustrating the~weighted admissibility for diagonal systems.

\section{Admissibility for diagonal semigroups}
\label{sec:admissibility}

Let $H$ be a Hilbert space and let $(\mathbb{T}_t)_{t \geq0}$ be a \emph{strongly continuous} (or $C_0$-) \emph{semigroup} of bounded linear operators on $H$ with the \emph{infinitesimal generator} $A : \mathcal{D}(A) \longrightarrow H$, defined by
\begin{displaymath}
Ax := \lim_{t \rightarrow 0^+} \frac{\mathbb{T}_t x - x}{t}, \qquad \text{ where } \qquad \mathcal{D}(A) := \left\{x \in X \; : \; \lim_{t \rightarrow 0^+} \frac{\mathbb{T}_t x - x}{t} \text{ exists} \right\}.
\end{displaymath}

Let us consider the linear system
\begin{equation}
\dot{x}(t) = Ax(t)+Bu(t), \qquad \qquad x(0) = x_0, \qquad \qquad t \geq 0,
\label{eq:}
\end{equation}
where $u(t) \in \mathbb{C}$ is the \emph{input} at time $t$ and $B : \mathbb{C} \longrightarrow \mathcal{D}(A^*)'$ is the \emph{control operator}. Here $\mathcal{D}(A^*)'$ is the completion of $H$ with respect to the norm
\begin{displaymath}
	\|x\|_{\mathcal{D}(A^*)'} := \left\|(\beta-A)^{-1} x \right\|_H,
\end{displaymath}
for any fixed $\beta \in \rho(A)$ (the resolvent set of $A$). An~operator $B$ is said to be \emph{finite-time $L^2$-admissible}, if for every $\tau >0$ and all $u \in L^2[0, \, \infty)$ the Bochner integral $\int_0^\tau \mathbb{T}_{\tau-t} Bu(t) \, dt$ lies in $H$ (see Definition 4.2.1, p. 116 in \cite{tucsnak2009}). Consequently, there exists $m_\tau>0$ such that
\begin{displaymath}
	\left\|\int_0^\tau \mathbb{T}_{\tau-t} Bu(t) \, dt \right\|_H \leq m_\tau \|u\|_{L^2[0, \, \infty)} \qquad \quad (\forall u \in L^2[0, \, \infty)).
\end{displaymath}
(see Proposition 4.2.2, also in \cite{tucsnak2009}). The admissibility criterion guarantees that the equation \eqref{eq:} has continuous (mild) solution 
\begin{displaymath}
	x(\tau) = \mathbb{T}_\tau x_0 + \int_0^\tau \mathbb{T}_{\tau-t} Bu(t) \, dt\qquad \qquad (\forall \tau \geq 0),
\end{displaymath}
with values in $H$ (see Proposition 4.2.5 in \cite{tucsnak2009}). If the constant $m_\tau$ can be chosen independently of $\tau>0$, then we say that $B$ is \emph{(infinite-time) $L^2$-admissible}. It follows that $B$ is $L^2$-admissible if and only if there exists a~constant $m_0>0$ such that
\begin{displaymath}
	\left\|\int_0^\infty \mathbb{T}_t Bu(t) \, dt \right\|_H \leq m_0 \|u\|_{L^2[0, \, \infty)} \qquad \qquad (\forall u \in L^2[0, \, \infty)),
\end{displaymath}
(see Remark 4.6.2 in \cite{tucsnak2009} and Remark 2.2 in \cite{haak2009}). This is a necessary condition for the \emph{state} $x(t)$ to lie in $H$. For more details see for example \cite{jacob2004}, and for the non-Hilbertian analogue: \cite{haak2010}, \cite{weiss1989}, \cite{weiss1989'}.

We may also consider the system
\begin{displaymath}
	\dot{x}(t) = Ax(t), \qquad \qquad y(t) = Cx(t), \qquad \qquad x(0) = x_0,
\end{displaymath}
where $C : \mathcal{D}(A) \longrightarrow \mathbb{C}$ is an \emph{$A$-bounded observation operator}, that is there exist $m_1, \, m_2 >0$ such that
\begin{displaymath}
	\|Cx\| \leq m_1 \|x\| + m_2 \|Ax\| \qquad \qquad (\forall x \in \mathcal{D}(A)).
\end{displaymath}
The operator $C$ is said to be \emph{(infinite-time) admissible} if there exists $m_0>0$ such that
\begin{displaymath}
	\|y\|_{L^2[0, \, \infty)} \leq m_0 \|x_0\|_H.
\end{displaymath}

There is duality between these two conditions, namely: $B$ is an admissible control operator for $(\mathbb{T}_t)_{t\geq 0}$ if and only if $B^*$ is an admissible observation operator for the dual semigroup $(\mathbb{T}^*_t)_{t\geq 0}$.

More detailed treatment of admissibility of control and observation semigroup operators and the theory of well-posed linear evolution equations is given the survey\cite{jacob2004} and the book \cite{tucsnak2009}.

For diagonal semigroups (see Example 2.6.6 in \cite{tucsnak2009}) the admissibility condition is linked to the theory of Carleson measures in the following way (see \cite{ho1983}, \cite{weiss1988}). Suppose that $A$ has a Riesz basis of eigenvectors $(\phi_k)_{k\in \mathbb{N}}$ (that is, there exists invertible $Q \in \mathscr{L}(H, \ell^2)$ such that $Q\phi_k = e_k$, for all $k \in \mathbb{N}$, where $(e_k)_{k\in \mathbb{N}}$ is the standard basis for $\ell^2$, i.e. each $e_k$ has 1 as its $k$-th entry and zeros elsewhere), with eigenvalues $(\lambda_k)_{k\in \mathbb{N}}$, each of them lying in the open left complex half-plane $\mathbb{C}_- := \left\{z \in \mathbb{C} \; : \; \re(z)<0\right\}$. Then a scalar control operator $B$, corresponding to a sequence $(b_k)_{k\in \mathbb{N}}$, is admissible if and only if the measure
\begin{displaymath}
	\mu := \sum_{k=1}^\infty |b_k|^2 \delta_{-\lambda_k}
\end{displaymath}
is a Carleson measure for the Hardy space $H^2(\mathbb{C}_+)$ on the right complex half-plane, that means the canonical embedding $H^2(\mathbb{C}_+) \longrightarrow L^2(\mathbb{C}_+, \, \mu)$ is bounded. An extension to normal semigroups has also been made in \cite{weiss1999}.

In some applications, requiring the input $u$ to lie $L^2[0, \, \infty)$ might be unsuitable, and hence a more general setting ought to be considered. In \cite{haak2005} and \cite{wynn2010} a concept of $\alpha$-admissibility, in which $u$ must lie in weighted $L^2_{t^\alpha}(0, \,\infty)$, for $\alpha>-1$ in the first, and $-1 < \alpha < 0$ in the latter article, was studied. The second paper linked admissibility to the Carleson measures, using the fact that the Laplace transform maps $L^2_{t^\alpha}(0, \, \infty)$ onto a~weighted Bergman space - in this article we adopt a similar approach. Papers \cite{haak2007} and \cite{haak2010} discuss the same problem in non-Hilbertian setting. Further generalisations, to $L^2_w(0, \, \infty)$, for any positive measurable weight $w$ and unweighted $L^p(0, \, \infty), \, 1 \leq p < \infty$, have been obtained in \cite{jacob2014}. In this paper we shall present it for the weighted $L^p_w(0, \, \infty)$ case, for certain weights (measurable selfmaps on $(0, \, \infty)$) $w$, and by the the weighted $L^p_w(0, \, \infty)$ we mean the Banach space of all functions $f : (0, \, \infty) \rightarrow \mathbb{C}$ satisfying
\begin{displaymath}
	\|f\|_{L^p_w(0, \, \infty)} := \left( \int_0^\infty |f(t)|^p w(t) \, dt \right)^{1/p} < \infty \qquad \qquad (1 \leq p < \infty).
\end{displaymath}

Given $1 \leq q < \infty$, assume that the semigroup $(\mathbb{T}_t)_{t \geq 0}$ acts on a Banach space $X$ with a $q-$Riesz basis (having the same definition as above, but with $\mathscr{L}(H, \ell^2)$ replaced by $\mathscr{L}(X, \ell^q)$), of eigenvectors $(\phi_k)_{k \in \mathbb{N}}$ with corresponding eigenvalues $(\lambda_k)_{k \in \mathbb{N}}\subset \mathbb{C}_-$; that is
\begin{displaymath}
	\mathbb{T}_t \phi_k = e^{\lambda_k t} \phi_k \qquad \qquad (\forall k \in \mathbb{N}),
\end{displaymath}
Suppose that $(\phi_k)_{k \in \mathbb{N}}$ is also a Schauder basis of $X$ such that there exist constants $c, C >0$ such that
\begin{displaymath}
	c \sum_{k=1}^\infty |a_k|^q \leq \left\|\sum_{k=1}^\infty a_k \phi_k \right\|^q \leq C \sum_{k=1}^\infty |a_k|^q \qquad \qquad (\forall (a_k)_{k=1}^\infty \in \ell^q).
\end{displaymath}
This means that we can effectively identify $X$ with $\ell^q$ and this shall be our standing assumption for the whole paper.

The following two theorems, proved in \cite{jacob2014}, link admissibility of control and observation operators with Laplace--Carleson embeddings (that is, Carleson embeddings induced by the Laplace transform). These results were presented there for weighted $L^2$ spaces and unweighted $L^p$ spaces on $(0, \, \infty)$, but it is easy to check that their proofs remain valid even for weighted $L^p$ spaces.

\begin{theorem}[Theorem 2.1 in \cite{jacob2014}]
\label{thm:1stjacob}
Let $1 \leq p, \, q  < \infty$, and suppose $X$ is defined as above. Let $w$ be a measurable self-map on $(0, \, \infty)$, and let $B$ be a bounded linear map from $\mathbb{C}$ to $\mathcal{D}(A^*)'$ corresponding to the sequence $(b_k)_{k \in \mathbb{N}}$. The control operator $B$ is $L^p_w$-admissible for $(\mathbb{T}_t)_{t \geq 0}$, that is, there exists a constant $m_0>0$ such that
\begin{displaymath}
	\left\|\int_0^\infty \mathbb{T}_t Bu(t) \, dt \right\|_X \leq m_0 \|u\|_{L^p_w(0, \, \infty)} \defeq m_0 \left(\int_0^\infty |u(t)|^p w(t)  \, dt \right)^{1/p},
\end{displaymath}
for all $u \in L^p_w(0, \, \infty)$, if and only if the Laplace transform induces a continuous mapping from $L^p_w(0, \, \infty)$ into $L^q(\mathbb{C}_+, \, \mu)$, where $\mu$ is the measure $\sum_{k=1}^\infty |b_k|^q \delta_{-\lambda_k}$.
\end{theorem}

Note that for $1<p<\infty$, we can associate the dual space of $L^p_w(0, \, \infty)$ with $L^{p'}_{w^{-p'/p}}(0, \, \infty)$, where $p':=p/(p+1)$ is the conjugate index of $p$ via the pairing
\begin{displaymath}
	\left\langle f, \, g \right\rangle = \int_0^\infty f(t)g(t) \, dt \qquad \qquad (f \in L^p_w(0, \, \infty), \, g \in L^{p'}_{w^{-p'/p}}(0, \, \infty))
\end{displaymath}
(see Remark 1.4 in \cite{haak2007} with $t^\alpha$ replaced by $w^{1/p}$), and hence the following result follows.

\begin{theorem}[Theorem 2.2 in \cite{jacob2014}]
\label{thm:2ndjacob}
Let $C$ be a bounded linear map from $\mathcal{D}(A)$ to $\mathbb{C}$. The observation operator $C$ is $L^p_w$-admissible for $(\mathbb{T}_t)_{t \geq 0}$, that is, there exists a constant $m_0>0$ such that
\begin{displaymath}
	\left\|C\mathbb{T}_. x \right\|_{L^p_w(0, \, \infty)} \defeq \left(\int_0^\infty |C\mathbb{T}_tx(t)|^p w(t)\, dt \right)^{1/p} \leq m_0 \|x\|_X \qquad (\forall x \in \mathcal{D}(A)),
\end{displaymath}
if and only if the Laplace transform induces a continuous mapping from \\ $L^{p'}_{w^{-p'/p}}(0, \, \infty)$ into $L^{q'}(\mathbb{C}_+, \, \mu)$, where $\mu$ is the measure $\sum_{k=1}^\infty |c_k|^{q'} \delta_{-\lambda_k}$, \\ $c_k := C\phi_k$, for all $k \in \mathbb{N}$, and $q':=q/(q-1)$ is the conjugate index of $q$.
\end{theorem}

So in order to test admissibility of a control operator we need to determine when the embedding
\begin{displaymath}
\mathfrak{L} : L^p_w(0, \, \infty) \longrightarrow L^q(\mathbb{C}_+, \, \mu) \qquad f \mapsto \mathfrak{L}f \defeq \int_0^\infty f(t) e^{-t \cdot} \, dt
\end{displaymath}
is bounded. Or, in other words, whether there exists a constant $C>0$ such that
\begin{displaymath}
	\left(\int_{\mathbb{C}_+} \left|\mathfrak{L}f\right|^q \, d\mu\right)^{p/q} \leq C \int_0^\infty |f(t)|^p w(t) \, dt \qquad \qquad (\forall f \in L^p_w(0, \, \infty)).
\end{displaymath}
If $p=q$ and this embedding is indeed bounded, then we shall refer to $\mu$ as a \emph{Carleson measure} for $\mathfrak{L}(L^p_w(0, \, \infty))$ (a space which we equip with the $L^p_w$ norm). Because the observation operator version of this problem is analogous, from now on we shall only state our results for control operators, leaving the observation operator case to be derived from duality by an interested reader.

\section{Carleson measures for Hilbert spaces of analytic functions on $\mathbb{C}_+$}
\label{sec:kernels}
Throughout this section the operator $B$ and the measure $\mu$ will be as defined in Section \ref{sec:admissibility}. We begin by considering the most elementary case, that is when $p~=~q~=~2$. Then $\mathfrak{L}(L^2_w(0, \, \infty))$, equipped with the $L^2_w$ inner product, is a Hilbert space of analytic functions. If for all $z \in \mathbb{C}_+$, $e^{-t\overline{z}}/w(t)$ belongs to $L^p_w(0, \, \infty)$, we can easily verify that it is also a reproducing kernel Hilbert space:
\begin{displaymath}
	\mathfrak{L}f(z) \defeq \int_0^\infty f(t) e^{-tz} \, dt = \int_0^\infty f(t) \overline{\frac{e^{-t\overline{z}}}{w(t)}} w(t) \, dt \defeq \left\langle \mathfrak{L}f, \, \mathfrak{L}\frac{e^{-t\overline{z}}}{w(t)}\right\rangle_{\mathfrak{L}(L^2_w(0, \, \infty))},
\end{displaymath}
for each $f \in L^2_w(0, \, \infty)$. Suppose that for each  $(z, \, \zeta) \in \mathbb{C}_+^2$, $k_z(\zeta)$ is the reproducing kernel of $\mathfrak{L}(L^2_w(0, \, \infty))$. In this case Lemma 24 from \cite{arcozzi2008} can be rephrased as the following proposition

\begin{prop}
The control operator $B$ is $L^2_w$-admissible if and only if the linear map
\begin{displaymath}
	f \mapsto \int_{\mathbb{C}_+} \re(k_z(\cdot)) f(z) \, d\mu(z) = \sum_{k=1}^\infty |b_k|^2 \re(k_{-\lambda_k}(\cdot))f(-\lambda_k)
\end{displaymath}
is bounded on $L^2(\mathbb{C}_+, \, \mu)$.
\label{prop:arcozziprop}
\end{prop}

\begin{corollary} ~
\begin{enumerate}
	\item If
	\begin{equation}
		\sum_{k=1}^\infty \sum_{l=1}^\infty \left| b_k b_l\re \left( k_{-\lambda_k} (-\lambda_l) \right) \right|^2 < \infty,
		\label{eq:strongercondition}
	\end{equation}
	then $B$ is $L^2_w$-admissible.
	\item If $B$ is $L^2_w$-admissible, then there exists $C>0$ such that
		\begin{displaymath}
		 \sum_{k \in \Gamma} \sum_{l \in \Gamma} \left|b_k b_l\re\left(k_{-\lambda_k}(-\lambda_l) \right)\right|^2 \leq C \sum_{n \in \Gamma} |b_n|^2 \qquad \qquad 	(\forall \Gamma \subset \mathbb{N}).
	\end{displaymath}
\end{enumerate}
\end{corollary}

\begin{proof}
To prove 1. we notice that by H\"{o}lder's inequality
\begin{align*}
	\int_{\mathbb{C}_+} &\left| \int_{\mathbb{C}_+} \re\left(k_z(\zeta) \right) G(z) \, d\mu(z) \right|^2 \, d\mu(\zeta) \\
	&\leq \left(\int_{\mathbb{C}_+} |G|^2 \, d\mu\right)\left( \int_{\mathbb{C}_+} \int_{\mathbb{C}_+} \left|\re\left(k_z(\zeta) \right)\right|^2 \, d\mu(z) d\mu(\zeta)\right),
\end{align*}
for all $G \in L^2(\mathbb{C}_+, \, \mu)$, and the result follows from the previous proposition. Also, by the proof of Lemma 26 from \cite{arcozzi2008}, we have
\begin{displaymath}
	\int_{\mathbb{C}_+} \int_{\mathbb{C}_+} \re\left(k_z\right) G(z) \overline{G(\zeta)} \, d\mu(z) d\mu(\zeta) \leq C(\mu) \int_{\mathbb{C}_+} |G|^2 \, d\mu, \qquad(\forall G \in L^2(\mathbb{C}_+, \, \mu))
\end{displaymath}
And then we apply it to $G=\chi_\Omega$, the characteristic function of $\Omega=\{-\lambda_k\}_{k \in \Gamma}$.
\end{proof}

If $\mathfrak{L}(L^2_w(0, \, \infty))$ is a Banach algebra with respect to the pointwise multiplication (for example $\mathfrak{L}(L^2_{1+t^2}(0, \, \infty))$, then, by Theorem 3 from \cite{kucik2016}, we know that $\mathfrak{L}(L^2_w(0, \, \infty))$ must also be a reproducing kernel Hilbert space (with kernel $k_z$, say) and
\begin{equation}
	\sup_{z \in \mathbb{C}_+} \|k_z\|_{\mathfrak{L}(L^2_w(0, \, \infty))} \leq 1.
	\label{eq:kernelcondition}
\end{equation}

\begin{prop}
Suppose that $\mathfrak{L}(L^2_w(0, \, \infty))$ is a Banach algebra with respect to the pointwise multiplication. If $(b_k)_{k=1}^\infty \in \ell^2$, then $B$ is $L^2_w$-admissible
\end{prop}

\begin{proof}
By the Cauchy--Schwarz inequality we get
\begin{align*}
\sum_{k=1}^\infty \sum_{l=1}^\infty \left| b_k b_l \re \left( k_{-\lambda_k} (-\lambda_l) \right) \right|^2 &\defeq \int_{\mathbb{C}_+} \int_{\mathbb{C}_+} \left|\re\left(k_z(\zeta) \right)\right|^2 \, d\mu(z) d\mu(\zeta) \\
&\leq \int_{\mathbb{C}_+} \int_{\mathbb{C}_+} \left|k_z(\zeta) \right|^2 \, d\mu(z) d\mu(\zeta) \\
&\leq \int_{\mathbb{C}_+} \int_{\mathbb{C}_+} \left\|k_z \right\|^2_{\mathfrak{L}(L^2_w(0, \, \infty))} \left\|k_\zeta \right\|^2_{\mathfrak{L}(L^2_w(0, \, \infty))} d\mu(z) d\mu(\zeta) \\
&\leq \left(\int_{\mathbb{C}_+} \left\|k_z \right\|^2_{\mathfrak{L}(L^2_w(0, \, \infty))} \, d\mu(z)\right)^2 \\
&\leq \left(\sup_{z \in \mathbb{C}_+} \left\|k_z \right\|^2_{\mathfrak{L}(L^2_w(0, \, \infty))} \int_{\mathbb{C}_+} \, d\mu\right)^2 \\
&\leq \left(\sum_{k=1}^\infty |b_k|^2 \right)^2 < \infty,
\end{align*}
and the result follows from the previous corollary.
\end{proof}
More examples of $\mathfrak{L}(L^2_w(0, \, \infty))$ can be easily produced from criteria given for example in \cite{kucik2016}, \cite{kuznetsova2006} or \cite{nikolskii1970}.

Let us now consider another type of spaces of analytic functions on $\mathbb{C}_+$. Let $\tilde{\nu}$ be a positive regular Borel measure on $[0, \, \infty)$ satisfying so-called \emph{$\Delta_2$-condition}:
\begin{equation}
	\sup_{r>0} \frac{\tilde{\nu}[0, \, 2r)}{\tilde{\nu}[0, \, r)}  < \infty,
		\tag{$\Delta_2$}
	\label{eq:delta2}
\end{equation}
and let $\lambda$ denote the Lebesgue measure on $i\mathbb{R}$. We define $\nu$ to be the positive regular Borel measure $\tilde{\nu} \otimes \lambda$ on $\overline{\mathbb{C}_+} := [0, \, \infty) \times i\mathbb{R}$. For this measure and $1 \leq p < \infty$, a \emph{Zen space} (see \cite{jacob2013}) is defined to be:
\begin{displaymath}
	A^p_\nu := \left\{F : \mathbb{C}_+ \longrightarrow \mathbb{C} \, \text{analytic} \; : \; \left\|F\right\|^p_{A^p_\nu} := \sup_{\varepsilon >0} \int_{\mathbb{C}_+} |F(z+\varepsilon)|^p \, d\nu < \infty \right\}.
\end{displaymath}
The Zen spaces generalise Hardy spaces on $\mathbb{C}_+$ (these correspond to $\tilde{\nu} = \frac{1}{2\pi}\delta_0$) and weighted Bergman spaces (corresponding to $d\tilde{\nu}=r^\alpha dr, \, \alpha >-1$). 
If $p=2$, then Zen spaces are Hilbert spaces (see \cite{peloso2016}), and in fact a reproducing kernel Hilbert spaces (see \cite{kucik2016}). In \cite{jacob2013} it was proved that the Laplace transform defines an isometric map
\begin{displaymath}
	\mathfrak{L} : L^2_w(0, \, \infty) \longrightarrow A^2_\nu,
\end{displaymath}
where $w$ is given by
\begin{displaymath}
	w(t) := 2\pi \int_0^\infty e^{-2rt} \, d\tilde{\nu}(r) \qquad \qquad (t > 0).
\end{displaymath}
The article \cite{jacob2013} also contains a full characterisation of Carleson measures for Zen spaces, which was also presented in terms of admissibility in \cite{jacob2014}. In \cite{kucik2016} a generalisation of Zen spaces was defined, namely
\footnotesize
\begin{displaymath}
	A^p\left( \mathbb{C}_+, \, (\nu_n)_{n=0}^m \right) := \left\{F : \mathbb{C}_+ \longrightarrow \mathbb{C} \, \text{analytic} \, : \, \left\|F\right\|^p_{A^p\left( \mathbb{C}_+, \, (\nu_n)_{n=0}^m \right)} := \sum_{n=0}^m \left\|F^{(n)}\right\|^p_{A^p_{\nu_n}} < \infty \right\},
\end{displaymath}
\normalsize
where each $\nu_n=\tilde{\nu}_n \otimes \lambda$, and $\tilde{\nu}_n$ is defined as $\nu$ above, and $0\leq m \leq \infty$. It is also proved there that if $p=2$, then the Laplace transform again defines an isometric map
\begin{displaymath}
	\mathfrak{L} : L^2_{w_{(m)}}(0, \, \infty) \longrightarrow A^2\left( \mathbb{C}_+, \, (\nu_n)_{n=0}^m \right),
\end{displaymath}
where $w$ is given by
\begin{displaymath}
	w_{(m)}(t) := 2\pi \sum_{n=0}^m t^{2n} \int_0^\infty e^{-2rt} \, d\tilde{\nu}_n(r)  \qquad \qquad (t >0).
\end{displaymath}
The image of $L^2_{w_{(m)}}(0, \, \infty)$ is denoted by $A^2_{(m)}$. This is a large class of spaces, containing for example Hardy and weighted Bergman spaces mentioned earlier, but also the Dirichlet space on $\mathbb{C}_+$ (when we have $\tilde{\nu_0}=\frac{1}{2\pi}\delta_0$ and $\tilde{\nu}_1$ being the Lebesgue measure with the weight $1/\pi$) which has not been studied often in the complex half-plane context before. And our problem of determining admissibility of control or observation operators is reduced to the characterisation of Carleson measures for $A^2_{(m)}$, allowing us to consider $L^2_w$-admissibility for non-decreasing weights, which were not included in the Zen space context. This has been partially done in \cite{kucik2017} and we aim to extend the results obtained there to the non-Hilbertian case of $A^p\left( \mathbb{C}_+, \, (\nu_n)_{n=0}^m \right)$ in the next section.

\section{$A^p(\mathbb{C}_+, \, (\nu_n)_{n=0}^m) \hookrightarrow L^q(\mathbb{C}_+, \, \mu)$ embeddings}
\label{sec:Ap}

The boundedness of canonical embeddings into $L^q(\mathbb{C}_+, \, \mu)$ (in this context also called \emph{Carleson embeddings}), for some Borel measure $\mu$, and characterisations of Carleson measures is very often given in terms of \emph{Carleson squares} (sometimes called Carleson boxes). On the half-plane these are defined as follows. In this section we prove general version of Theorems 2 and 4 from \cite{kucik2017}. Note that for $p=2$ these can be used to describe corresponding $L^2_{w_{(m)}}$-admissibility. This is left for the interested reader.

\begin{definition}
Let $a \in \mathbb{C}_+$. A \emph{Carleson square centred at $a$} is defined to be the set
\begin{equation}
Q(a) := \left\{z = x+iy \; : \; 0 \leq x < 2\re(a), \, |y-\im(a)| \leq \re(a)\right\}.
\label{eq:carlesonsquare}
\end{equation}
\end{definition}

\begin{theorem}
Suppose that $m < \infty$. If the embedding 
\begin{displaymath}
	A^p(\mathbb{C}_+, \, (\nu_n)_{n=0}^m) \hookrightarrow L^q(\mathbb{C}_+, \, \mu)
\end{displaymath}
is bounded, then there exists a constant $C(\mu)>0$, such that
\begin{equation}
	\mu(Q(a)) \leq C(\mu) \left[\sum_{n=0}^m \frac{\nu_n\left(\overline{Q(a)}\right)}{(\re(a))^{np}}\right]^{\frac{q}{p}},
	\label{eq:necessarycarlesonforAp}
\end{equation}
for each Carleson square $Q(a)$.
\label{thm:necessary}
\end{theorem}

\begin{proof}
Let $a \in \mathbb{C}_+$, and choose $\gamma > \sup_{0\leq n \leq m}(\log_2 R_n -np +1)/p$, where $R_n$ denotes the supremum obtained from the \eqref{eq:delta2}-condition for each $\tilde{\nu}_n, \, 0~\leq~n~\leq~m$.  Then for all $z$ in $Q(a)$ we have $|z+\overline{a}| \leq \sqrt{10}\re(a)$, and hence
\begin{equation}
	\frac{\mu\left(Q(a)\right)}{(\sqrt{10}\re(a))^{\gamma q}} \leq \int_{\mathbb{C}_+} \frac{d\mu(z)}{|z+\overline{a}|^{\gamma q}}.
\label{eq:muestimate}
\end{equation}
Similarly
\begin{displaymath}
	|z+\overline{a}| \geq \sqrt{\re(a)^2} = \re(a) > \frac{\re(a)}{2} \qquad \qquad (\forall z \in Q(a)).
\end{displaymath}
Also, given $k \in \mathbb{N}_0$, for all $z \in Q(2^{k+1}\re(a)+i\im(a)) \setminus Q(2^k(\re)+i\im(a))$, with $0 < \re(z) \leq 2^{k+1}\re(a)$ we have
\begin{displaymath}
	|z+\overline{a}| \geq \sqrt{\re(a)^2+ (2^k\re(a))^2} \geq 2^k \re(a),
\end{displaymath}
and even if $2^{k+1}\re(a) < \re(z) \leq 2^{k+2}\re(a)$, we also have
\begin{displaymath}
	|z+\overline{a}| \geq \sqrt{(2^{k+1}\re(a)+\re(a))^2} \geq 2^k \re(a).
\end{displaymath}
And
\begin{align}
\begin{split}
	\nu_n&\left(Q(2^{k+1}\re(a)+i\im(a)) \setminus Q(2^k(\re)+i\im(a))\right) \leq \nu_n\left(Q(2^{k+1}\re(a)+i\im(a))\right) \\
	&\leq \tilde{\nu}_n \left[0, \, 2^{k+2}\re(a)\right) \cdot 2^{k+1}\re(a) \stackrel{\eqref{eq:delta2}}{\leq} \left(2R_n\right)^{k+1} \tilde{\nu}_n [0, \, 2\re(a)) \cdot 2\re(a) \\
	&\leq \left(2R_n\right)^{k+1} \nu_n\left(\overline{Q(a)}\right),
\end{split}
\label{eq:delta2fornu}
\end{align}
so
\begin{align*}
	\int_{\mathbb{C}_+} \frac{d\nu_n(z)}{|z + \overline{a}|^{(\gamma+n)p}} &\leq \left(\frac{2}{\re(a)}\right)^{(\gamma+n)p} \nu_n\left(Q(a)\right) \\
	&\qquad + \sum_{k=0}^\infty \frac{\nu_n\left(Q(2^{k+1}\re(a)+i\im(a)) \setminus Q(2^k(\re(a))+i\im(a))\right)}{(2^k\re(a))^{(\gamma+n)p}} \\
	&\stackrel{\eqref{eq:delta2fornu}}{\leq} \left(\frac{2}{\re(a)}\right)^{(\gamma+n)p} \nu_n\left(\overline{Q(a)}\right) \left(1+\sum_{k=0}^\infty \frac{(2R_n)^{k+1}}{2^{(k+1)(\gamma+n)p}}\right) \\
	&\leq \left(\frac{2}{\re(a)}\right)^{(\gamma+n)p} \nu_n\left(\overline{Q(a)}\right) \sum_{k=0}^\infty \left(\frac{R_n}{2^{(\gamma+n)p-1}}\right)^k
\end{align*}
and the sum converges for all $0 \leq n \leq m$. Hence $(z+\overline{a})^{-\gamma} \in A^p(\mathbb{C}_+, \, (\nu_n)_{n=0}^m)$. Now, if the embedding is bounded, with constant $C'(\mu)>0$ say, then
\begin{align*}
\mu(Q(a)) &\stackrel{\eqref{eq:muestimate}}{\leq} (\sqrt{10}\re(a))^{\gamma q} \int_{\mathbb{C}_+} \frac{d\mu(z)}{\left|z+\overline{a}\right|^{\gamma q}} \\
&\leq C'(\mu)(\sqrt{10}\re(a))^{\gamma q} \left[\sum_{n=0}^m \int_{\mathbb{C}_+} \frac{d\nu_n(z)}{\left|\left[\left(z+\overline{a}\right)^\gamma \right]^{(n)}\right|^p}\right]^{\frac{q}{p}} \\
&\leq C'(\mu)(\sqrt{10}\re(a))^{\gamma q} \left[\sum_{n=0}^m \left(\prod_{l=1}^n(\gamma+l-1) \right)\int_{\mathbb{C}_+} \frac{d\nu_n(z)}{|z+\overline{a}|^{(\gamma+n)p}}\right]^{\frac{q}{p}} \\
&\leq C(\mu) \left[\sum_{n=0}^m \frac{\nu_n\left(\overline{Q(a)}\right)}{(\re(a))^{np}}\right]^{\frac{q}{p}},
\end{align*}
where
\begin{displaymath}
C(\mu) := 2^{q(n+3\gamma/2)}5^{\gamma q/2} \left[\left(\prod_{l=0}^m (\gamma+l-1)\right) \max_{0\leq n \leq m} \sum_{k=0}^\infty \left(\frac{R_n}{2^{(\gamma+n)p-1}}\right)^k\right]^{\frac{q}{p}} \, C'(\mu),
\end{displaymath}
(and we adopted the convention that the product $\prod(\gamma+l-1)$ is defined to be 1, if the lower limit is a bigger number than the upper limit).
\end{proof}

\begin{remark}
For $p=q$ and $m=0$ (i.e. Carleson measures for Zen spaces), this result was stated in \cite{jacob2014}, and proved to be necessary as well as sufficient. An extension to $A^2(\mathbb{C}_+, \, (\nu_n)_{n=0}^m)$ was made in \cite{kucik2017}, but only as a necessary condition. In the last section of this paper we will prove that for some sequences of measures $(\nu_n)_{n=0}^m$ and sectorial measures $\mu$ it is also sufficient. However it still remains unclear if this could be true for a general case.
\end{remark}

A version of the next theorem (for Carleson measures for $A^2_{(m)})$ has been proved in \cite{kucik2017}, following closely earlier version for analytic Besov spaces in \cite{arcozzi2002} on the open unit disk $\mathbb{D}$ of the complex plane and Drury--Averson Hardy spaces and other Besov-Sobolev spaces on complex balls from \cite{arcozzi2008}.

\begin{theorem}
\label{thm:arcozzi}
Let $1<p \leq q < \infty$ and let $\mu$ be a positive Borel measure on $\mathbb{C}_+$. If $\rho$ is a regular weight such that 
\begin{equation}
	\int_{\mathbb{C}_+} |F'(z)|^p (\re(z))^{p-2} \rho(z)\, dz \leq \|F\|^p_{A^p(\mathbb{C}_+, \, (\nu_n)_{n=0}^m)},
	\label{eq:besov}
\end{equation}
for all $F \in A^p(\mathbb{C}_+, \, (\nu_n)_{n=0}^m)$ and there exists a constant $C(\mu, \rho)>0$ such that
\begin{equation}
\left(\int_{Q(a)} \frac{(\mu(Q(a) \cap Q(z))^{p'}}{(\re(z))^2} \rho(z)^{1-p'} \, dz \right)^{q'/p'} \leq C(\mu, \rho) \mu(Q(a)) \qquad (\forall a \in \mathbb{C}_+),
\label{eq:1}
\end{equation}
then the embedding
\begin{displaymath}
	A^p(\mathbb{C}_+, \, (\nu_n)_{n=0}^m) \hookrightarrow L^q(\mathbb{C}_+, \, \mu)
\end{displaymath}
is bounded.
\end{theorem}

\begin{proof}
Let $\zeta \in \mathbb{C}_+$. We use a representation of $\mathbb{C}_+$ as an ordered tree $T(\zeta)$, namely, we decompose the complex half-plane into a set of rectangles
\begin{displaymath}
	R_{(k,l)}(\zeta) := \left\{z \in \mathbb{C}_+ \; : \; 2^{k-1} < \frac{\re(z)}{\re(\zeta)} \leq 2^{k}, \, 2^k l \leq \frac{\im(z)-\im(\zeta)}{\re{\zeta}} < 2^k (l+1)\right\},
\end{displaymath}
for all $(k, \, l) \in \mathbb{Z}^2$, and we identify each of these rectangles with a vertex of an abstract graph $T(\zeta)$. We put an order relation "$\leq$" on the set of vertices of $T(\zeta)$ by saying that $x \leq y$ whenever the area of the rectangle corresponding to $x$ is greater or equal to the area of the rectangle corresponding to $y$ and there is a sequence of horizontally adjacent rectangles $(R_{k, \, l}(\zeta))$ forming a path connecting the rectangles corresponding to $x$ and $y$. This decomposition is detailed in \cite{kucik2017}. 

Given $F \in A^p(\mathbb{C}_+, \, (\nu_n)_{n=0}^m)$, for each $\alpha \in T(\zeta)$ let $w_\alpha, \, z_\alpha \in \overline{\alpha} \subset \mathbb{C}_+$ be such that
\begin{displaymath}
	z_\alpha := \sup_{z \in \alpha} \{|F(z)|\} \qquad \text{ and } \qquad w_\alpha := \sup_{w \in \alpha} \{|F'(w)|\}.
\end{displaymath}
Define a weight $\tilde{\rho}$ on $T(\zeta)$ by $\tilde{\rho}(\alpha) := \rho(z_\alpha)$. And also: $r_\alpha = \re(w_\alpha)/4, \, \Phi(\alpha):= F(z_\alpha), \, \varphi(\alpha) = \Phi(\alpha)-\Phi(\alpha^-)$, for all $\alpha \in T(\zeta)$. Note that $\mathcal{I}\varphi = \Phi$. This is becuase if $F$ is in $A^p(\mathbb{C}_+, \, (\nu_n)_{n=0}^m$, then it is in the Zen space $A^p_{\nu_0}$, and hence in the Hardy space $H^p(\mathbb{C}_+)$ (or its shifted version, see \cite{peloso2016}), and hence $\lim_{\alpha \longrightarrow - \infty} |F(z_\alpha)| = \lim_{\re(z) \longrightarrow \infty}|F(z)|=0$. Since \eqref{eq:1} holds, we can apply Lemmata 3 and 4 from \cite{kucik2017} to $\varphi, \, \tilde{\rho}, \, \tilde{\mu}$ (where $\tilde{\mu}(\alpha):= \mu(\alpha)$, for all $\alpha \in T(\zeta)$) in the following way
\begin{align*}
\int_{\mathbb{C}_+} |F|^q \, d\mu & = \sum_{\alpha \in T(\zeta)} \int_\alpha |F|^q \, d\mu \leq \sum_{\alpha \in T(\zeta)} |\Phi(\alpha)|^q \tilde{\mu}(\alpha) \\
&\lessapprox \left(\sum_{\alpha \in T(\zeta)} |\varphi(\alpha)|^p \tilde{\rho}(\alpha)\right)^{q/p} \defeq \left(\sum_{\alpha \in T(\zeta)} |\Phi(\alpha)-\Phi(\alpha^-)|^p \tilde{\rho}(\alpha)\right)^{q/p} \\
&\stackrel{\stackrel{\text{Fundamental Thm}}{\text{of Calculus}}}{\leq} \left(\sum_{\alpha \in T(\zeta)} \left| \int_{z_{\alpha^-}}^{z_\alpha} F'(w) \, dw \right|^p \tilde{\rho}(\alpha)\right)^{q/p} \\
&\lessapprox \left(\sum_{\alpha \in T(\zeta)} \text{diam}(\alpha)^p |F'(w_\alpha)+F'(w_{\alpha^-})|^p \tilde{\rho}(\alpha)\right)^{q/p} \\
&\lessapprox \left(\sum_{\alpha \in T(\zeta)} \text{diam}(\alpha)^p |F'(w_\alpha)|^p \tilde{\rho}(\alpha)\right)^{q/p} \\
&\stackrel{\stackrel{\text{Mean-Value}}{\text{Property}}}{\leq} \left(\sum_{\alpha \in T(\zeta)} \text{diam}(\alpha)^p \left|\frac{1}{\pi r^2_\alpha} \int_{B(w_\alpha, \, r_\alpha)} F'(z) \, dz \right|^p \tilde{\rho}(\alpha)\right)^{q/p} \\
&\stackrel{\text{H\"{o}lder's}}{\leq} \left(\sum_{\alpha \in T(\zeta)} \frac{\text{diam}(\alpha)^p}{(\pi r_\alpha^2)^{p(1-1/p')}} \int_{B(w_\alpha, \, r_\alpha)} |F'(z)|^p \, dz \tilde{\rho}(\alpha)\right)^{q/p} \\
&\lessapprox \left(\sum_{\alpha \in T(\zeta)} \text{diam}(\alpha)^{p-2} \int_{\bigcup_{\beta \in T(\zeta) \; : \; \beta \cap B(w_\alpha, \, r_\alpha) \neq \varnothing}} |F'(z)|^p  \, dz \, \tilde{\rho}(\alpha) \right)^{q/p} \\
&\lessapprox \left(\sum_{\alpha \in T(\zeta)} \int_{\bigcup_{\beta \in T(\zeta) \; : \; \beta \cap B(w_\alpha, \, r_\alpha) \neq \varnothing}} |F'(z)|^p  \frac{\rho(z)}{(\re(z))^{2-p}}\, dz  \right)^{q/p} \\
&\lessapprox \left(\sum_{\alpha \in T(\zeta)} \int_{\alpha} |F'(z)|^p \frac{\rho(z)}{(\re(z))^{2-p}} \, dz\right)^{q/p},
\end{align*}
which is less than $\|F\|^q_{A^p(\mathbb{C}_+, \, (\nu_n)_{n=0}^m)}$ by the assumption of the theorem.
\end{proof}

\begin{remark}
Although condition \eqref{eq:besov} looks unnecessarily artificial and very restrictive, it simply means that $A^p(\mathbb{C}_+, \, (\nu)_{n=0}^m)$ is (or is contained within) some analytic Besov space on $\mathbb{C}_+$. For example, if $p=2$ and $\rho \equiv 1$, then $A^p(\mathbb{C}_+, \, (\nu)_{n=0}^m) \subseteq \mathcal{D}(\mathbb{C}_+)$, the Dirichlet space on $\mathbb{C}_+$. Condition \eqref{eq:1}, expressed in terms of Carleson boxes on $\mathbb{D}$ and distance from $\partial \mathbb{D}$, is known to be necessary and sufficient for the disk equivalent of the above theorem. It is not clear whether the same could be true for $\mathbb{C}_+$.
\end{remark}

\section{Laplace--Carleson embeddings for sectorial measures}
\label{sec:sectorial}

Testing the boundedness of Laplace--Carleson embedding for arbitrary \\ $1 \leq p, \, q < \infty$ is generally very difficult. We can however obtain some partial results if we consider measures with some restrictions imposed on their support

\begin{prop}
Let $1<p< \infty, \, 1 \leq q < \infty$, let $w$ be a measurable self-map on $(0, \, \infty)$ and suppose that $\mu$ be a positive Borel measure supported on $(0, \infty)$. If the Laplace--Carleson embedding $\mathfrak{L} : L^p_w (0, \, \infty) \longrightarrow L^q (\mathbb{C}_+, \, \mu)$ is well-defined and bounded, then
	\begin{displaymath}
		\mu(I) \leq C \left( \int_0^\infty \frac{e^{-|I|p't}}{w^{\frac{1}{p-1}}(t)} \, dt \right)^{-\frac{q}{p'}},
	\end{displaymath}
	for all intervals $I = (0, \, |I|]$, provided that the integral on the right exists.
	\end{prop}
	
\begin{proof}
	\item Let $0 < x \leq |I|$ and $a>0$, then
	\begin{displaymath}
		\left|\mathfrak{L}\left[\frac{e^{-\cdot a}}{w^{\frac{1}{p-1}}(\cdot)}\right](x)\right| = \int_0^\infty \frac{e^{-t(a+x)}}{w^{\frac{1}{p-1}}(t)} \, dt \geq \int_0^\infty \frac{e^{-t(a+|I|)}}{w^{\frac{1}{p-1}}(t)} \, dt.
	\end{displaymath}
	And hence
	\begin{align*}
		\mu(I) &\leq \left(\int_0^\infty \frac{e^{-t(a+|I|)}}{w^{\frac{1}{p-1}}(t)} \, dt\right)^{-q} \int_I \left|\mathfrak{L}\left[\frac{e^{-\cdot a}}{w^{\frac{1}{p-1}}(\cdot)}\right](x)\right|^q \, d\mu(x) \\
		&\leq \left(\int_0^\infty \frac{e^{-t(a+|I|)}}{w^{\frac{1}{p-1}}(t)} \, dt\right)^{-q} \int_{\mathbb{C}_+} \left|\mathfrak{L}\left[\frac{e^{-\cdot a}}{w^{\frac{1}{p-1}}(\cdot)}\right](x)\right|^q \, d\mu(x) \\
		&\leq C(\mu) \left(\int_0^\infty \frac{e^{-t(a+|I|)}}{w^{\frac{1}{p-1}}(t)}\, dt\right)^{-q} \left\|\frac{e^{-\cdot a}}{w^{\frac{1}{p-1}}(\cdot)}\right\|^q_{L^p_w (0, \, \infty)} \\
		&= C(\mu) \left(\int_0^\infty \frac{e^{-t(a+|I|)}}{w^{\frac{1}{p-1}}(t)}\, dt\right)^{-q} \left(\int_0^\infty \frac{e^{-apt}}{w^{\frac{p}{p-1}}(t)} w(t) \, dt \right)^{\frac{q}{p}} \\
		&= C(\mu) \left(\int_0^\infty \frac{e^{-t(a+|I|)}}{w^{\frac{1}{p-1}}(t)}\, dt\right)^{-q} \left(\int_0^\infty \frac{e^{-apt}}{w^{\frac{1}{p-1}}(t)} \, dt \right)^{\frac{q}{p}},
	\end{align*}
	where $C(\mu)>0$ is the constant from the Laplace--Carleson embedding. Choosing $a = |I|/(p-1)$ gives us the desired result.
	\end{proof}
	
\begin{theorem}
Given $0<a\leq b< \infty$, let
\begin{displaymath}
	S_{(a,\, b]} := \{z \in \mathbb{C}_+ \; : \; a < \re(z) \leq b\}.
\end{displaymath}
If there exists a partition 
\begin{displaymath}
	P : 0 < \ldots \leq x_{-n} \leq \ldots \leq x_{-1} \leq x_0 \leq x_1 \leq \ldots \leq x_n \leq \ldots \qquad \qquad n \in \mathbb{N}
\end{displaymath}
of $(0, \, \infty)$ and sequence $(c_n) \in \ell^1_{\mathbb{Z}}$ such that
	\begin{displaymath}
		\mu(S_{(x_n, \, x_{n+1}]}) \leq |c_n| \left(\int_0^\infty \frac{e^{-p'tx_n}}{w^{\frac{1}{p-1}}(t)} \, dt \right)^{-\frac{q}{p'}} \qquad \qquad (\forall n \in \mathbb{Z}),
	\end{displaymath}
	then the Laplace--Carleson embedding $\mathfrak{L} : L^p_w (0, \, \infty) \longrightarrow L^q (\mathbb{C}_+, \, \mu)$ is well-defined and bounded.
	\end{theorem}

\begin{proof}
	For any$z \in S_{(x_k, \, x_{k+1}]}$ and $f \in L^p_w(0, \, \infty)$ we have
	\begin{displaymath}
		|\mathfrak{L}f(z)| \leq \int_0^\infty e^{-tx_n} |f(t)| \, dt \stackrel{\text{H\"{o}lder's}}{\leq} \left(\int_0^\infty \frac{e^{-p'tx_n}}{w^{\frac{1}{p-1}}(t)} \, dt \right)^{\frac{1}{p'}} \|f\|_{L^p_w (0, \, \infty)},
	\end{displaymath}
	so 
	\begin{align*}
	\int_{\mathbb{C}_+} |\mathfrak{L}f|^q \, d\mu &\leq \|f\|^q_{L^p_w (0, \, \infty)} \sum_{n=-\infty}^\infty \left(\int_0^\infty \frac{e^{-p'tx_n}}{w^{\frac{1}{p-1}}(t)} \, dt \right)^{\frac{q}{p'}} \mu(S_{(x_n, \, x_{n+1}]}) \\
	&\leq \|(c_n)\|_{\ell^1_{\mathbb{Z}}} \|f\|^q_{L^p_w (0, \, \infty)}.
	\end{align*}
\end{proof}

\begin{definition}
Let $1 \leq p \leq \infty$ and let $f \in L^p(\mathbb{R})$. We define the \emph{maximal function} of $f$ to be
\begin{displaymath}
	Mf(x) = \sup_{r>0} \frac{1}{2r} \int_{|y| \leq r} |f(x-y)| \, dy.
\end{displaymath}
\end{definition}
The maximal function of $f$ is finite almost everywhere. The book \cite{stein1993} by E. M. Stein offers extensive description of the maximal function and its properties, such as the link between $Mf$ and the $L^p$ norm of $f$ used in the arguments below.

\begin{lemma}
Let $f$ be in $L^p_w(0, \, \infty), \, 1 \leq p < \infty$. Then for all $x>0$ and any partition 
\begin{displaymath}
	P \, : \, 0 \leq \ldots \leq t_{-k} \leq \ldots \leq t_0 = 1 \leq t_1 \leq \ldots \leq t_k \leq \ldots \qquad k \in \mathbb{N}_0
\end{displaymath}
of $[0, \, \infty)$, such that $\inf_{k \in \mathbb{N}} t_{-k} =0$. We then have
\begin{equation}
	\int_0^\infty e^{-\frac{t}{x}} |f(t)| \, dt  \leq \Theta(P, w, x) xMg(x),
\label{eq:estlemma}
\end{equation}
where 
\begin{displaymath}
g(t) = \begin{cases} w^{1/p}(t)f(t), &\quad t >0, \\ 0 &\quad t\leq 0, \end{cases}
\end{displaymath}
$g \in L^p(\mathbb{R})$, and
\begin{displaymath}
	\Theta(P, w, x) = 2\left[\sum_{k=-\infty}^{-1} \frac{e^{-t^*_k}}{w^{\frac{1}{p}}(t^*_k x)} \left(1-t_k\right) + \sum_{k=0}^{\infty} \frac{e^{-t^*_k}}{w^{\frac{1}{p}}(t^*_k x)} \left(t_{k+1}-1\right)\right],
\end{displaymath}
where each $t^*_k$ is such that
\begin{displaymath}
	 \frac{e^{-t^*_k}}{w^{\frac{1}{p}}(t^*_k x)} \geq  \frac{e^{-t}}{w^{\frac{1}{p}}(t x)} \qquad \qquad (\forall t \in \left(t_k, \, t_{k+1} \right)).
\end{displaymath}
\label{estlemma}
\end{lemma}

\begin{proof}
Let $r_k := \max \left\{|1-t_k|, \, |1-t_{k+1}|\right\}$, for each $k$. Given $x>0$, we have
\begin{align*}
\int_0^\infty e^{-\frac{t}{x}} |f(t)| \, dt &= x \int_0^\infty e^{-t} |f(tx)| \, dt \leq x\sum_{k=-\infty}^\infty \frac{e^{-t^*_k}}{w^{\frac{1}{p}}(t^*_k x)} \int_{t_k}^{t_{k+1}} |g(tx)| \, dt \\
&=\sum_{k=-\infty}^\infty \frac{e^{-t^*_k}}{w^{\frac{1}{p}}(t^*_k x)} \int_{(1-t_{k+1})x}^{(1-t_k)x} |g(x-y)| \, dy \\
&\leq \sum_{k=-\infty}^\infty \frac{e^{-t^*_k}}{w^{\frac{1}{p}}(t^*_k x)} \frac{r_k x}{r_k x} \int_{|y| \leq r_k x} |g(x-y)| \, dy \\
&\leq 2\left[\sum_{k=-\infty}^\infty \frac{e^{-t^*_k}}{w^{\frac{1}{p}}(t^*_k x)} r_k \right] xMg(x).
\end{align*}
To get the required result note that if $k \leq -1$, then $t_{k+1} \leq t_0 = 1$ and hence
\begin{displaymath}
 1-t_k \geq 1-t_{k+1} \geq 0 \quad \Longrightarrow \quad r_k=|1-t_k| = 1-t_k,
\end{displaymath}
otherwise $t_{k+1} > 1$, so $t_k \geq 1$, and so
\begin{displaymath}
 0 \geq 1-t_k \geq 1-t_{k+1} \quad \Longrightarrow \quad r_k=|1-t_{k+1}| = t_{k+1}-1.
\end{displaymath}
\end{proof}

The following theorem has been proved in \cite{jacob2013} (Theorem 3.3, p. 801) for unweighted $L^p(0, \, \infty)$ case. We use the above lemma to obtain a weighted version.

\begin{theorem}
\label{thm:sectorial}
Let $1 < p \leq q < \infty$, let $\mu$ be a positive Borel measure on $\mathbb{C}_+$ supported only in the sector
\begin{displaymath}
	S(\theta) := \{z \in \mathbb{C}_+ \, : \, |\arg(z)| < \theta\},
\end{displaymath}
for some $0 \leq \theta < \pi/2$, and let $\alpha<p-1$. For  an interval $I=(0, \, |I|) \subset \mathbb{R}$ we define
\begin{displaymath}
	\Delta_I: = \left\{z \in S(\theta) \; : \; \re(z) \leq |I| \right\}.
\end{displaymath}
The Laplace--Carleson embedding $\mathfrak{L} : L^p_{t^\alpha}(0, \, \infty) \longrightarrow L^q(\mathbb{C}_+, \, \mu)$ is well-defined and bounded if and only if there exists a constant $C(\mu) >0$ such that
\begin{equation}
	\mu(\Delta_I) \leq C(\mu) |I|^{\frac{q}{p'}\left(1-\frac{\alpha}{p-1}\right)}
	\label{eq:seclapcarl}
\end{equation}
for all intervals $I=(0, \, |I|) \subset \mathbb{R}$.
\end{theorem}

\begin{proof}
Suppose first that \eqref{eq:seclapcarl} holds. Let 
\begin{displaymath}
	T_n := \left\{z \in S(\theta) \; : \; 2^{n-1} < \re(z) \leq 2^n \right\} \subset \Delta_{(0, \, 2^n)} \qquad  \qquad (n \in \mathbb{Z}),
\end{displaymath}
and let also $x_n = 2^{-n+1}$. Clearly
\begin{displaymath}
	S(\theta) = \bigcup_{n \in \mathbb{Z}} T_n \qquad \text{ and } \qquad \mu(T_n) \leq \mu(\Delta_{(0, \, 2^n)}) \stackrel{\eqref{eq:seclapcarl}}{\leq} C(\mu) x_n^{-\frac{q}{p'}\left(1-\frac{\alpha}{p-1}\right)}.
\end{displaymath}
By the previous lemma we have that
\begin{displaymath}
	|\mathfrak{L}f(z)| \leq \int_0^\infty e^{-\frac{t}{x_n}} |f(t)| \, dt \stackrel{\eqref{eq:estlemma}}{\leq} \Theta(P, t^\alpha, x_n) x_n Mg(x_n),
\end{displaymath}
for all $z \in T_n$ ($\Theta$ and $g$ are as in Lemma \ref{estlemma}). Note that the choice of $t_k^*$ does not depend on $x_n$, since
\begin{displaymath}
	\frac{e^{-t^*_k}}{(t^*_k x_n)^{\frac{\alpha}{p}}} \geq  \frac{e^{-t}}{(t x_n)^{\frac{\alpha}{p}}} \; \; \forall t \in (t_k, \, t_{k+1}) \quad \Longleftrightarrow \quad \frac{e^{-t^*_k}}{(t^*_k)^{\frac{\alpha}{p}}} \geq  \frac{e^{-t}}{t^{\frac{\alpha}{p}}}  \; \; \forall t \in (t_k, \, t_{k+1}),
\end{displaymath}
and there exists a partition $P$ of $(0, \, \infty)$, for which $\Theta(P, t^\alpha, x_n)$ converges (since $\alpha < p$), so fixing $P$ we can define $D_\Theta = x_n^{\frac{\alpha}{p}} \Theta(P, t^\alpha, x_n)$, which, by the definition of $\Theta$, is a constant depending on $P$ and $\alpha$ only. Thus we have
\begin{align*}
\int_{S(\theta)} |\mathfrak{L}f|^q \, d\mu &\leq D_\Theta \sum_{n=-\infty}^\infty \left(x_n^{1-\frac{\alpha}{p}} Mg(x_n)\right)^q \mu(T_n) \\
&\leq C(\mu)D_\Theta \sum_{n=-\infty}^\infty x_n^{q\left(1-\frac{\alpha}{p}\right)-\frac{q}{p'}\left(1-\frac{\alpha}{p-1}\right)} Mg(x_n)^q \\
&= C(\mu)D_\Theta \sum_{n=-\infty}^\infty x_n^{q\left(1-\frac{\alpha}{p}-\frac{1}{p'}+\frac{\alpha}{p}\right)} Mg(x_n)^q \\
&= C(\mu)D_\Theta \sum_{n=-\infty}^\infty \left(x_n Mg(x_n)^p\right)^{\frac{q}{p}} \\
&\leq C(\mu)D_\Theta \left(\sum_{n=-\infty}^\infty x_n Mg(x_n)^p\right)^{\frac{q}{p}} \\
&\lessapprox \left\|g\right\|^q_{L^p(0, \, \infty)} = \left\|f\right\|^q_{L^p_{t^\alpha}(0, \, \infty)}.
\end{align*}
Now suppose that the converse is true. For each $z \in \Delta_I$ we have $|z| \leq |I|\sec(\theta)$, so
	\begin{displaymath}
		\left|\mathfrak{L} \left[\frac{e^{-|I|\sec(\theta)t}}{t^{\frac{\alpha}{p-1}}}\right](z)\right| = \frac{\Gamma\left(1-\frac{\alpha}{p-1}\right)}{\left|z+|I|\sec(\theta)\right|^{1-\frac{\alpha}{p-1}}} \geq  \frac{\Gamma\left(1-\frac{\alpha}{p-1}\right)}{\left(2|I|\sec(\theta)\right)^{1-\frac{\alpha}{p-1}}}.
	\end{displaymath}
	And therefore we have
	\begin{align*}
		\mu(\Delta_I) &\lessapprox |I|^{q\left(1-\frac{\alpha}{p-1}\right)} \int_{S(\theta)} \left|\mathfrak{L} \left[\frac{e^{-|I|\sec(\theta)t}}{t^{\frac{\alpha}{p-1}}}\right](z)\right|^q \, d\mu(z) \\ 
		&\lessapprox |I|^{q\left(1-\frac{\alpha}{p-1}\right)} \left\|\frac{e^{-|I|\sec(\theta)t}}{t^{\frac{\alpha}{p-1}}}\right\|^q_{L^p_{t^\alpha}(0, \, \infty)} \\
		&= |I|^{q\left(1-\frac{\alpha}{p-1}\right)} \left(\int_0^\infty \frac{e^{-|I|p\sec(\theta)t}}{t^{\frac{\alpha}{p-1}}} \, dt \right)^{\frac{q}{p}} \\
		&\lessapprox |I|^{q\left(1-\frac{\alpha}{p-1}\right)} |I|^{-\frac{q}{p}\left(1-\frac{\alpha}{p-1}\right)} \\
		&= |I|^{\frac{q}{p'}\left(1-\frac{\alpha}{p-1}\right)},
		\end{align*}
		as required.
\end{proof}

\begin{corollary}
Let $1 < p \leq q < \infty$, let $w$ be a measurable self-map on $(0, \, \infty)$, and let $\mu$ be a positive Borel measure on $\mathbb{C}_+$ supported only in the sector $S(\theta)$, $0~\leq~\theta~<~\pi/2$. Suppose that \begin{displaymath}
	\sup_{t>0} \frac{t^\alpha}{w(t)} < \infty
\end{displaymath}
for some $\alpha < p-1$. If for some family of intervals $(I_n)_{n \in \mathbb{Z}} = ((0, \, 2^n|I|))_{n \in \mathbb{Z}}$ there exists $C(\mu)>0$ such that
	\begin{displaymath}
	\mu(\Delta_{I_n}) \leq C(\mu) (|I_n|)^{\frac{q}{p'}\left(1-\frac{\alpha}{p-1}\right)} \qquad \qquad (\forall n \in \mathbb{Z}),
	\end{displaymath}
then the Laplace--Carleson embedding $\mathfrak{L} : L^p_w (0, \, \infty) \longrightarrow L^q (\mathbb{C}_+, \, \mu)$ is well-defined and bounded.
\end{corollary}

\begin{proof}
By the previous theorem we get that
\begin{displaymath}
\int_{S(\theta)} |\mathfrak{L}f|^q \, d\mu \lessapprox \left\|f\right\|^q_{L^p_{t^\alpha}(0, \, \infty)} \leq \left(\sup_{t>0}\frac{t^\alpha}{w(t)}\right)^{\frac{q}{p}} \left\|f\right\|^q_{L^p_w(0, \, \infty)}.
\end{displaymath}
\end{proof}

\begin{corollary}
Let $B$ and $\mu$ be defined as in Theorem \ref{thm:1stjacob}, let $1<p \leq q < \infty$ and $\alpha < p-1$, and suppose that there exists $0 < \theta < \pi/2$ such that
\begin{displaymath}
	\im(-\lambda_k) < \re(-\lambda_k)\tan \theta \qquad \qquad (\forall k \in \mathbb{N}).
\end{displaymath}
Then the control operator {B} is $L^p_{t^\alpha}$-admissible if and only if there exists a constant $C(\mu)>0$ such that
\begin{displaymath}
	\sum_{k \in \Gamma} |b_k|^q \leq C(\mu) \max_{k \in \Gamma}\left[\re(-\lambda_k)\right]^{\frac{q}{p'}\left(1-\frac{\alpha}{p-1}\right)} \qquad \qquad (\forall \Gamma \subset \mathbb{N}).
\end{displaymath}
\label{cor:seccor}
\end{corollary}

\begin{example}
Consider the following one-dimensional heat PDE on the interval $[0, \, 1]$:
\begin{displaymath}
\begin{cases}
	\frac{\partial z}{\partial t} (\zeta, \, t) = \frac{\partial^2 z}{\partial \zeta^2} (\zeta, \, t) \\ 
	\frac{\partial z}{\partial \zeta} (0, \, t) = 0 \\
	\frac{\partial z}{\partial \zeta} (1, \, t) = u(t) \\
	 z(\zeta,0) = z_0(\zeta)
	\end{cases} \qquad \qquad \zeta \in (0, 1), \, t \geq 0.
\end{displaymath}
This system can be expressed in the form \eqref{eq:} with $H=\ell^2, \, Ae_n = -n^2\pi^2e_n$, and $b_n = 1$, for each $n \in \mathbb{N}$ (see Example 3.6 in \cite{jacob2014}). For $1<p \leq 2$ and $\alpha < p-1$, by the previous corollary we have that $B$ is $L^p_{t^\alpha}$-admissible if and only if $p \geq \frac{4}{3}(\alpha+1)$.
\end{example}

\begin{example}
Let $1<p\leq 2$. Consider the following parabolic diagonal system: let $X~=~\ell^2, \, \lambda_n~=~-2^n, \, b_n~=~2^n/n$, and let $A$ be defined by $Ae_n~=~\lambda_n e_n, \, n~\in~\mathbb{N}$. By the previous corollary, if $\alpha \leq -1$, then the control operator $B$ is $L^2_{t^\alpha}$-admissible, and if $-1 < \alpha < p-1$, then $B$ is not $L^2_{t^\alpha}$-admissible. This contrasts with the unweighted setting, in which for all $1<p< \infty$, the control operator $B$ is not $L^p$-admissible. This was proved only very recently in \cite{jacob2016} (Example 5.2).
\end{example}

\begin{theorem} Let $\mu$ be a positive Borel measure supported only in the sector $S(\theta), \, 0 < \theta < \pi/2$. If there exists an interval $I \subset i\mathbb{R}$, centred at 0, and a constant $C(\mu)>0$ such that
	\begin{equation}
		\mu\left(Q(2^k|I|)\right) \leq C(\mu) \left[\left(\nu_0\left(Q(2^k|I|)\right)\right)^{-\frac{1}{2}} + \left(\sum_{n=0}^m \frac{\nu_n\left(Q(2^k|I|)\right)}{(2^k |I|)^{2n}} \right)^{-\frac{1}{2}}\right]^{-2},
		\label{eq:similarcon}
	\end{equation}
for all $k \in \mathbb{Z}$, then $\mu$ is a Carleson measure for $A^2_{(m)}$.
\end{theorem}

\begin{proof}
For all $t, \, x >0$ we have
\begin{align*}
w_{(m)}(tx) &\defeq 2\pi \sum_{n=0}^m (tx)^{2n} \int_0^\infty e^{-2rtx} \, d\tilde{\nu}_n(r) \\
&\geq 2\pi \sum_{n=0}^m t^{2n} 2^{2n} \left(\frac{x}{2}\right)^{2n} e^{-t} \tilde{\nu}_n\left[0, \, \frac{1}{2x}\right) \\
&\geq 2\pi \sum_{n=0}^m t^{2n} \left(\frac{x}{2}\right)^{2n} e^{-t} \frac{\tilde{\nu}_n\left[0, \, \frac{2}{x}\right)}{R^2_n},
\end{align*}
where each $R_n$ is the supremum obtained from the \eqref{eq:delta2}-condition, corresponding to $\tilde{\nu}_n$. Clearly we have that
\begin{displaymath}
	w_{(m)}(tx) \geq 2\pi e^{-t} \frac{\tilde{\nu}_0\left[0, \, \frac{2}{x}\right)}{R^2_0}, \qquad \qquad (\forall t, x > 0),
\end{displaymath}
and
\begin{displaymath}
	w_{(m)}(tx) \geq 2\pi \sum_{n=0}^m \left(\frac{x}{2}\right)^{2n} e^{-t} \frac{\tilde{\nu}_n\left[0, \, \frac{2}{x}\right)}{R^2_n}, \qquad \qquad (\forall x > 0, \, t \geq 1).
\end{displaymath}
Let
\begin{displaymath}
	P \, : \, 0 = \ldots = t_{-k}= \ldots  = t_{-1} <  t_0 = 1 \leq t_1 \leq \ldots \leq t_k \leq \ldots, \qquad \qquad (k \in \mathbb{N}),
\end{displaymath}
be a partition of $[0, \, \infty)$, and let $x_k = 2^{-k+1}|I|^{-1}, \, k \in \mathbb{Z}$. Then
\begin{align*}
	\Theta(P, \, w_{(m)}, x_k) &\defeq 2\left[\frac{e^{-t^*_{-1}}}{\sqrt{w_{(m)}(t^*_{-1} x_k)}} + \sum_{l=0}^{\infty} \frac{e^{-t^*_l}}{\sqrt{w_{(m)}(t^*_l x_k)}}(t_{l+1}-1)\right] \\
&\leq  \sqrt{\frac{2}{\pi}} \left[ \frac{R_0}{\sqrt{\tilde{\nu}_0\left[0, \, \frac{2}{x_k}\right)}} + \frac{\sum_{l=0}^\infty e^{-\frac{t_l}{2}}t_{l+1}}{\sqrt{\sum_{n=0}^m \left(\frac{x_k}{2}\right)^{2n} \frac{\tilde{\nu}_n\left[0, \, \frac{2}{x_k}\right)}{R^2_n}}} \right].
\end{align*}
And by Lemma \ref{estlemma} we get that for any $z \in T_k$
\begin{align*}
	|\mathfrak{L}f(z)| &\leq \sqrt{\frac{2}{\pi}} \left[ \frac{R_0}{\sqrt{\tilde{\nu}_0\left[0, \, \frac{2}{x_k}\right)}} + \frac{\sum_{l=0}^\infty e^{-\frac{t_l}{2}}t_{l+1}}{\sqrt{\sum_{n=0}^m \left(\frac{x_k}{2}\right)^{2n} \frac{\tilde{\nu}_n\left[0, \, \frac{2}{x_k}\right)}{R^2_n}}} \right] x_k Mg(x_k)\\
	&= \sqrt{\frac{2}{\pi}} \left[ \frac{R_0}{\sqrt{\frac{1}{x_k}\tilde{\nu}_0\left[0, \, \frac{2}{x_k}\right)}} + \frac{\sum_{l=0}^\infty e^{-\frac{t_l}{2}}t_{l+1}}{\sqrt{\sum_{n=0}^m \left(\frac{x_k}{2}\right)^{2n-1} \frac{\tilde{\nu}_n\left[0, \, \frac{2}{x_k}\right)}{R^2_n}}} \right] \sqrt{x_k} Mg(x_k)\\
	&\lessapprox \left[\left(\nu_0\left(Q(2^k|I|)\right)\right)^{-\frac{1}{2}} + \left(\sum_{n=0}^m \frac{\nu_n\left(Q(2^k|I|)\right)}{(2^k |I|)^{2n}} \right)^{-\frac{1}{2}}\right] \sqrt{x_k}Mg(x),
\end{align*}
so for any $\mathfrak{L}f=F \in A^2_{(m)}$ we have
\begin{displaymath}
	\int_{\mathbb{C}_+}|F|^2 \, d\mu = \int_{S(\theta)}|\mathfrak{L}f|^2 \, d\mu \lessapprox \sum_{k=-\infty}^\infty x_k (Mg(x_k))^2 \lessapprox \|f\|^2_{L^2_{w_{(m)}}(0, \infty)} = \|F\|^2_{A^2_{(m)}},
\end{displaymath}
as required.
\end{proof}

\begin{remark}
Note that condition \eqref{eq:similarcon}, although it looks somehow superficial, actually almost matches condition \eqref{eq:necessarycarlesonforAp} from Theorem \ref{thm:necessary} (with $m=1$ and $p=q=2$). It suggests that if there exists a criterion characterising Carleson measures for $A^2_{(m)}$, which is both necessary and sufficient, then it must be expressible in a very similar form. This, however, still remains to be done.
\end{remark}

\textbf{Acknowledgement} 	The author of this article would like to thank the UK Engineering and Physical Sciences Research Council (EPSRC) and the School of Mathematics at the University of Leeds for their financial support. He is also very indebted to Professor Jonathan R. Partington for all the invaluable comments and help in preparation of this research paper.

\end{document}